\documentclass[12pt]{article}
\usepackage{amsmath}

\usepackage{epsf}

\usepackage{indentfirst}
\usepackage{latexsym}
\usepackage{amssymb}
\usepackage[dvips]{graphicx}
\usepackage{flafter}
\usepackage{caption2}
\usepackage{textcomp}
\usepackage{supertabular}
\usepackage{longtable, booktabs}

\usepackage{multirow}

\usepackage{hhline}
\usepackage{bigstrut,bigdelim}

\usepackage{rotating}
\usepackage{multicol}
\usepackage{multirow}
\usepackage{makeidx} 
\usepackage{epsfig}
\usepackage{fancyhdr}       

\newif\ifger

\gerfalse

\newtheorem{theorem}{Theorem}[section]

\newtheorem{lemma}[theorem]{Lemma}

\newtheorem{corollary}[theorem]{Corollary}

\newtheorem{definition}[theorem]{Definition}

\newtheorem{proposition}[theorem]{Proposition}

\def\S{\mathcal{S}}
\def\P{\mathcal{P}}
\def\L{\mathcal{L}}
\def\C{{\mathfrak C}}

\def\D{\mathcal{D}}\def\G1{G^{\C}}
\newcommand{\lcm}{\mathrm{lcm}}

\newcommand{\cC}{{\mathfrak C}}

\newcommand{\cS}{{\cal S}}

\def\al{\alpha}
\def\be{\beta}
\def\ga{\gamma}
\def\de{\delta}
\def\lam{\lambda}

\def\PG{{\rm PG}}

\def\Soc{{\rm Soc}}
\def\Aut{{\rm Aut}}

\newcommand{\Fix}{{\rm Fix}}

\newcommand{\PSL}{{\rm PSL}}
\newcommand{\PGL}{{\rm PGL}}

\newcommand{\GL}{{\rm GL}}

\newcommand{\PSU}{{\rm PSU}}

\def\al{\alpha}

\def\al{\alpha}
\def\be{\beta}

\def\Fix{{\rm Fix}}

\def\PSL{{\rm PSL}}
\def\PGL{{\rm PGL}}

\def\Soc{{\rm Soc}}

\begin{document}

\title{Line-transitive
point-imprimitive linear spaces with Fang-Li parameter $\gcd(k,r)$
at most 10}

\author{Haiyan Guan, Delu Tian and Shenglin Zhou\footnote{Partially supported by the NSFC(No:11071081). Corresponding author:
slzhou@scut.edu.cn }\\
{\small\it  Department of Mathematics, South China University of Technology,}\\
 {\small\it Guangzhou, Guangdong 510640, P. R. China}
 }
\maketitle
\date{}

\begin{abstract} This paper is a further contribution to the classification of line-transitive
finite linear spaces. We prove that if $\S$ is a non-trivial finite linear space with the Fang-Li parameter $\gcd(k,r)$ is 9 or 10, the group $G\leq \Aut(\S)$ is line-transitive and point-imprimitive, then $\S$ is the Desarguesian projective plane
$\PG(2,9)$.


\medskip
\noindent{\bf MR(2000) Subject Classification} 05B05, 05B25, 20B25

\end{abstract}

\maketitle

\section{Introduction}

A finite linear space $\S$  is an incident structure $(\P,\L)$
 where $\P$ is a set of $v$ points and $\L$ is a set of $b$ distinguished subsets of $\P$ called lines, such that any two points are
incident with exactly one line. A linear space $\S$ is  said to be non-trivial if
 every line is incident with at least 3 points and $b \geq 2$.

 An automorphism of $\S$ is a  permutation  of $\P$ which leaves the set $\L$ invariant. The full automorphism group of $\S$ is denoted by $\Aut(\S)$. Any subgroup of $\Aut(\S)$ is called an automorphism group of $\S$.
  If a subgroup $G$ of $\Aut(\S)$ acts transitively on the set of points (lines, flags), then we say that $G$ is line-transitive (point-transitive, flag-transitive), where a flag is an incident point-line pair $(\al, \lam)$.
 Similarly, the automorphism group $G$ is said to be point-primitive (line-primitive) if it acts primitively on
 points (lines). If the sizes of all lines are
equal, then we say that $\S$ is a regular linear space.
 In this paper, we assume that the automorphism group $G \leq \Aut(\S)$ is
 line-transitive. It follows that $\S$ is regular, and the line size will be denoted by $k$. It is well-known that if
 $G $ is line-transitive then it is also point-transitive(\cite{Block1967}), and for a non-trivial regular linear space every point lies on the same number $r$ of lines.

In \cite{BDLPZ2}, the line-transitive point-imprimitive linear spaces with the Fang-Li parameter
$\gcd(k,r)\leq 8$ are classified. As an application, C. E. Praeger and S. L. Zhou \cite{PZ2007} give the
 classification of the line-transitive linear spaces with $k\leq 12$, extending the
Camina-Mischke's one in \cite{CaminaMischke96}.

This paper is a continuation of  \cite{BDLPZ2}. We extended the
classification of line-transitive poin-imprimitive linear spaces
with the Fang-Li parameter $\gcd(k,r)$ up to 10, and proved the following result.

\begin{theorem}
  Let $\S$ be a non-trivial finite linear space with $v$ points
and Fang-Li parameter $k^{(r)}=9$ or $10$.  Assume that $G\leq \Aut(\S)$ is line-transitive and
point-imprimitive. Then $\S$ is the Desarguesian projective plane
$\PG(2,9)$.
\end{theorem}

The proof of Theorem 1.1 rely on the realize of the algorithms in \cite{BDLPZ} and the available theory of finite line-transitive, point-imprimitive linear spaces in \cite{BDLPZ2,PraegerTuan,PZ2007}.

Combining Theorem 1.1 and \cite[Theorem 1.2]{BDLPZ2}, we get the following.

\begin{corollary}
Suppose that $\cS$ is a linear space  admitting a
line-transitive, point-imprimitive subgroup of automorphisms, and with the Fang-Li parameter  $k^{(r)}\leq10$. Then $\cS$ is the
Desarguesian projective plane $\PG(2,4)$, $\PG(2,7)$ or $\PG(2,9)$, or the
Mills design or Colbourn-McCalla design, both with
$(v,k)=(91,6)$, or one of the $467$ Nickel-Niemeyer-O'Keefe-Penttila-Praeger designs with $(v,k)=(729,8)$.
\end{corollary}

The structure of the paper is as follows. Section 2 contains some
parameters and our hypotheses, Section 3 gives some elementary
results on linear spaces which will be frequently used in the proof of our result.  In Section 4, we used and realized the algorithms
in \cite{BDLPZ} to produce a list of potential parameters and some
group theoretic information for $k^{(r)}=9$ or $10$. We then made a
detailed analysis of each of these potential parameters in Sections
 5 and 6.

\section{Definitions and Hypothesis}
In this paper, we use the following definitions, notation and hypothesis.

\subsection{Dlandtsheer-Doyen
parameters}

 If $G$ acts point-imprimitively on $\P$, then $\P$ has a non-trivial partition of imprimitivity. Let  $\C=
 \{C_1, C_2,\cdots, C_d\}$ be the partition where $|C_i|=c$ for all
 $1\leq i\leq d$, so that $v=cd$. Then there exist two positive integers $x,y$, such that\begin{align*}\label{eqn:DD}
c = \frac{\binom{k}{2}-x}{y} \quad \text{and} \quad
d=\frac{\binom{k}{2}-y}{x},
\end{align*} where $x$ is the number of inner pairs on a line (\cite{DelandtsheerDoyen89}). Here
$c,d,x,y$ are called the Dlandtsheer-Doyen parameters.

\subsection{Fang-Li parameters}

 Let $v,b,k$ be
the parameters of $\S=(\P, \L)$,
 and $r=\frac{v-1}{k-1}$ is the number of lines incident with a given point which is independent of the choice of point by  the line-transitivity. Define
 $$k^{(v)} =\gcd(k,v),k^{(r)}=\gcd(k,r), b^{(v)}=\gcd(b,v), b^{(r)}= \gcd(b,r).$$
Then $$b = b^{(v)}b^{(r)}, k=k^{(v)}k^{(r)}, v=b^{(v)}k^{(v)},
r=k^{(r)}b^{(r)}. $$ Clearly, $k^{(r)}=k/\gcd(k,v)=\gcd(k,v-1)$. Here $ k^{(v)} ,k^{(r)},b^{(v)},b^{(r)}$ are
called the Fang-Li parameters (\cite{FL1991}).

\subsection{Top and Bottom groups}

 Let $C\in \C$,
 $G^{\C}$  be the transitive permutation group on $\C$
induced by $G$, then $G^{\C}=G/G_{(\C)}$ and it is called the top
group(\cite{BDLPZ}). Similarly, define the bottom group $G^C=G_C/G_{(C)}$ (\cite{BDLPZ}), which is the transitive
group induced by $G$ on $C \in \C$.
Moreover, $G^C$ is independent of the choice of  $C$ in $\cC$ up to
permutation isomorphism. For partitions $\C, \C'$ of $\P$, $\C$
\emph{refines} $\C'$ if every class of $\C$ is contained in a class
of $\C'$, and this refinement is strict if $\C\ne\C'$. We also say
that $\C'$ is {\em coarser or strictly coarser} than $\C$. We call a
partition  minimal if it has no non-trivial strict refinement,
maximal if the only $G$-invariant partition of $\P$ that is strictly coarser than $\C$ is the trivial
partition $\{\P\}$ with only one class. If
a partition $\C$ of $\P$ is both  maximal and minimal, then we call
$G$  {\em 2-step imprimitive} relative to $\C$, in this case, both
$G^{\C}$ and $G^C$ are primitive. So if $G$ is 2-step imprimitive then $G^{\C}$ and  $G^C$ are primitive groups of degree $d$ and $c$ respectively. Relying on the classification of primitive groups of degree less than 4096 (cf. \cite{RD2,DM1996,RD}), if $c\leq 4096$ then we know the structure of the bottom group, similarly for $d$ and the top group. This will be frequently used in Sections 5 and 6.

 If the subgroup
$G_{(\C)}$ is transitive on each class of $\C$, then we say that the partition $\C$ is $G$-normal. By \cite{CP2001}, we
know that if $G$ is a line-transitive point-imprimitive group of
automorphisms of a linear space, then either \begin{enumerate}
  \item there exists a non-trivial $G$-normal point-partition; or
  \item $G$ is point-quasiprimitive and almost simple.
\end{enumerate}
 Recall that a finite permutation group $G$ is
{\em quasiprimitive} if every nontrivial normal subgroup of $G$ is
transitive. Equivalently, $G$ is quasiprimitive if every minimal
normal subgroup of $G$ is transitive (since overgroups of transitive
groups are transitive).

\subsection{Intersection type and $t_{max}$}

 The intersection type and $t_{max}$ will play an important role in our discussion. We give their definitions in the following.
\begin{definition}{\rm (\cite{AD1989})}\,\,
   Let $\S$ be a linear space, $\lam \in \L $. Define
 $$d_i=|{ C\in \C : |C \cap \lam| = i }|,   $$and  $$(0^{d_0} , 1^{d_1}, \cdots, k^{d_k}) $$ to be the intersection type of $\S$.

 We let the set of non-zero intersection sizes
 \begin{center}Spec$(\S):=\{i > 0 : d_i \neq 0 \}$.\end{center} is
 the spectrum of $\S$.
 Since the transitivity of $G$ on $\L$, the intersection type and spectrum  are independent of the choice of $\lam$.
 We sometimes write $Spec_{\C}(\S) $ if we need to specify the partition $\C$. Since $d_0=d-\sum_{i=1}^{k}d_i$, then it does not  matter to denote the intersection type of $\S$ by $( 1^{d_1}, \cdots, k^{d_k}) $.
\end{definition}

\begin{definition}{\rm (\cite{BDLPZ})}\,\,
For a given intersection type $(1^{d_1}, \cdots, k^{d_k})$, and
non-empty subset $S \subseteq Spec(\S)$, set $d(S)= \sum_{i\in
\S}{d_i}$. Define $t_{max}$  to be the largest positive integer t
such that, for all
 $S \subseteq $Spec$(\S)$, and all positive integers $h\leq min\{t, d(S)\}$,
$$\prod_{j=0}^{h-1}{(d-j)}\mid b\prod_{j=0}^{h-1}{(d(\S)-j)}.$$
\end{definition}

\subsection{Group theory notations and {\sc Hypothesis}}

If
$\C$ is $G$-normal, then let $K= G_ {(\C)}$, $S=\Soc(K)$, $X=C_G(K)$,
and $Y=C_G(S)$ (\cite{PZ2007}), where $\Soc(G)$ denotes the socle of the group $G$,
that is, the product of the minimal normal subgroups of $G$.
Throughout the paper we assume the following {\sc Hypotheses}.

\medskip
\textbf{\sc Hypothesis}

  Let $\S$ be a non-trivial finite linear space with $v$ points and $b$ lines, each of  size $k$, and
  with $r$ lines through each point. Let $G\leq \Aut(\S)$ be line-transitive and
  point-imprimitive such that $\C= \{C_1, C_2, \cdots, C_d\}$
  be a $G$-invariant partition of $\P$ with $d$ classes of size $c$ where
 $c>1$ and $d >1$.
 Also $\S$ has the Delandtsheer-Doyen parameters $c, d,
 x,y$, the Fang-Li parameters $k^{(v)}, k^{(r)}, b^{(v)}, b^{(r)}$,
 intersection type $(1^{d_1}, \cdots, k^{d_k})$, $t_{max}$, the
 subgroups $K$, $S$, $X$ and $Y$ if $\C$ is $G$-normal defined as above.

\section{Some Preliminary Results }

The following  lemma is essential for the proof of Theorem 1.1. So
for the completeness, we give a detailed proof here.
\begin{lemma}\label{p-k}{\rm \cite[Lemma 5]{CaminaMischke96}}\,
Let $G$ satisfy {\sc Hypothesis}, $p$ be a prime.
\begin{enumerate}
\item If $p\mid |G_{(\lam)}|$, and $k^2 -k +1 > max(r + k -p+1, r)$, then $p\mid |G^{\lam}|$ for any line $\lam$.
\item If $p > k$ and $k^2-k+1 > r$, then $p\mid v$ or $p\mid (v-1)$. Further if $T$ is a Sylow p-subgroup of $G$ then $|T|$
divides $v$ or $v-1$ respectively.\end{enumerate}
\end{lemma}

 \textbf{Proof:}
 (i) Let $p\mid |G_{(\lam)}|$ and  $P$ be a Sylow $ p$-subgroup of $G_{\lam}$ for $\lam \in \L$. If $p\nmid |G^{\lam}|$, then $P$ fixes $\lam$
 point-wise. So if $P$ fixes two points, it has to fix the unique line point-wise containing the two points.
 If the set of fixed points $F=Fix_{\P}(P)$ is the line $\lam$,
      by \cite[Lemma 3]{AJ1989}, we conclude that $G$ is flag-transitive, a contradiction. Thus $F$ has a
      structure of regular linear space with line size $k$, therefore, $k^2-k+1 \leq |F|$. Furthermore, by \cite[Lemma 1]{AJ1989}, $|F|\leq max(k+r-p-1, r)$, a desired contradiction.\\
   (ii) Let $P$ be a $p$-subgroup of $G$. If $P \leq G_{\lam}$ for $\lam \in \L$. By (i)  $p\mid |G^{\lam}|$, but  $G^{\lam}\leq S_k$, this is impossible. Thus $P = 1$. Therefore , there is no $p$-subgroup fixing a line, then at most one point is fixed by $P$, that is to say $|T|$
divides $v$ when it has no fixed point or $v-1$ when it has one fixed
point on $\P$. $\hfill\square$

Note that if $p\mid|G_{\lam}|$, (i) also holds.

For a primitive permutation group $G$, we call $G$ is $t$-transitive
if the integer $t$ is the maximal integer such that $G$ is
$t$-transitive.

\begin{lemma} \label{t}{\rm \cite[Lemma 4.7]{BDLPZ}} If $G^{\C}$ is t-transitive then $ t\leq t_{max}$.\end{lemma}
\begin{lemma}\label{123}{\rm \cite[Proposition 2.6]{AAC}} Assume that the {\sc Hypothesis} holds.
\begin{enumerate}
\item The number $b^{(r)}$ divides each non-trivial subdegree of $G^{\C}$, and in particular, $rank (\G1) \leq 1 + \de$.
\item The number  $b^{(r)}$  divides each non-trivial subdegree of $G^C$, and in particular, $rank( G^C) \leq 1 + \ga$.
\item Moreover, for $\al \in \P$, $b^{(r)}$  divides each non-trivial subdegree of $G^{\P}$ and each orbit length
of $G_{\al}$ in $\lam \in \L$ and $\al \in \lam$.

\end{enumerate}
\end{lemma}

\begin{lemma}\label{1231}Assume that the {\sc Hypothesis} holds.
\begin{enumerate}
\item {\rm\cite[Proposition 3.6]{LL2001}}\,\, If $\ga=2$, $b^{(r)}$ is odd and $c$ is not a prime power, then $G^C$ is 2-transitive on $C$.
\item {\rm \cite[Proposition 3.7]{LL2001}}\,\, If $\de=2$, $b^{(r)}$ is odd and $d$ is a prime power, then $G^{\C}$ is not 2-transitive on $\C$.
\end{enumerate}

\end{lemma}
\begin{lemma}\label{ngh}{\rm \cite[Lemma 2]{AJ1989}}\,\, Assume that $\S=(\P,\L)$ is a non-trivial linear
space admitting a line-transitive automorphism group $G$. Let $\lam
\in \L$ and $H \leq G_{\lam}$ such that, for $F :=
\Fix_{\P}(H)$,\begin{enumerate}
\item $ 2 \leq |F \cap \lam| < |F|$, and
\item if $K \leq G_{\lam} $ and $|\Fix_{\P}(K)\cap \lam| \geq 2$, and $H$ and $K$ are conjugate in $G$, then $H$ and $K$
are conjugate in $G_{\lam}$.
\end{enumerate}
Then the deduced linear space $\S|_F=(F, \L|_F)$ has constant line size and $N_G(H)$ acts line-transitively on $\S|_F$, where $\L|_F=\{\lam \cap F: \lam \in \L, |\lam\cap F|\geq2\}$.
\end{lemma}
\begin{lemma}\label{p3}{\rm \cite[Corollary 4.4]{BDLPZ}}\quad Assume that the {\sc Hypothesis} holds. Let $\lam \in \L$ and $p$ be a prime dividing
$|G_{\lam}|$. Let $P$ be a Sylow p-subgroup of $G_{\lam}$ and $F = \Fix_{\P}(P)$. Suppose that $2 \leq|F \cap \lam | < |F|$.Then
\begin{enumerate}
 \item $N_G(P)$ is line-transitive on $\D|_F$.
 \item $C|_F = \{C \cap F : C\in \C,C \cap F \neq 0\}$ is an $N_G(P)$-invariant partition of $F$.
 \item $|F|=f \cdot |C \cap F|$ where $f = |\C|_F |$, $C \cap F \in \C|_F$, and $|C \cap F|\geq 3$.
\end{enumerate}
\end{lemma}
\begin{lemma}\label{Kc}{\rm \cite[Theorem 6.1]{BDLPZ}}\quad Assume that the {\sc
Hypothesis} holds and that C is G-normal.\begin{enumerate}
\item If $k>2x+\frac{3}{2}+\sqrt{4x-\frac{7}{4}}$, then $G_{(\C)}$ is semiregular on points and lines, $|G_{(\C)}| = c$ is
odd, and $d_1> 0$.
\item If $x\leq 8$, then $G_{(\C)}$ has an abelian subgroup $S$ of index at most 2 such that $S$ is normal
in $G$, semiregular on points, and $|S| = c$ is odd.
\item If either of the conditions of (i) or (ii) holds, and if $\C$ is minimal, then $c$ is an odd
prime power and $ G^C$ is affine.
  \end{enumerate}
\end{lemma}

\begin{lemma} \label{lem:KSXY}{\rm \cite[Lemma 2.2]{PZ2007}}\quad   Assume that the {\sc Hypothesis} hold and that
$\C$ is $G$-normal and minimal.
\begin{enumerate}
\item[(a)] Then $\C$ is the set of $S$-orbits in $\mathcal{P}$ and
\begin{enumerate}
\item[(i)] Either  $Y\cap K=1$, or $S$ is elementary abelian and   $Y\cap
K=S$.
\item[(ii)] Either $X\cap K=1$, or $S$ is elementary abelian. and $X\cap K=K=S$.
\end{enumerate}
\item[(b)] Suppose in addition that $\C$ is maximal, that $Y\cap K=S\neq Y$,
and that one of the following conditions holds.
$$
\begin{array}{|c|c|c|}
\hline
\mbox{Condition} & \mbox{$\Soc(G^{\C})$}&\mbox{Extra Property}\\
\hline
1& \mbox{abelian} & \mbox{$\gcd(c,d)=1$}\\
2& \mbox{non-abelian} & \mbox{Schur multiplier of a minimal normal subgroup} \\
 &    & \mbox{of $G^{\C}$ has no section isomorphic to $ S $}\\
\hline
\end{array}
$$
Then $G$ has a  normal subgroup $M=T\times S$ where $T$ is a
minimal normal subgroup of $G$ and  $T^{\C}$ is
minimal normal in $G^{\C}$. Moreover either
\begin{enumerate}
\item[(i)] $T$ is non-abelian and transitive on $\mathcal{P}$,  or
\item[(ii)] the set $\C'$ of $T$-orbits in $\mathcal{P}$ is a $G$-normal
partition of $\mathcal{P}$ with $|\C'|=c$ such that for $C\in \C$ and
$C'\in \C'$, $|C'|=d$ and $|C\cap C'|=1$. Moreover either $M$ is
regular on $\mathcal{P}$ or $T$ is not semiregular on $\mathcal{P}$.
\end{enumerate}

\end{enumerate}
\end{lemma}

\section{Computing parameters}
Assuming the {\sc Hypothesis} holds with $k^{(r)}=9$ or $10$,
 we get totally 304
potential parameter sets (called {\sc Lines})
$$(d, c, x, y, \gamma, \delta, k^{(v)},
k^{(r)}, b^{(v)}, b^{(r)}, {\rm intersection\,\,type}, t_{\max}),$$ by applying Algorithms 1 and 2 in \cite{BDLPZ} which implemented in GAP(\cite{GAP}).
For each {\sc Line}, there is no other partition of $\mathcal{P}$ which is
neither a strict refinement nor strict coarser than $ \C,$ so the
automorphism group is 2-step imprimitive relative to $ \C$.

In Section 5, we deal with 33 {\sc Lines} with
$k^{(r)}=9$  listed in Table 1. In Section 6 the remaining 271  {\sc Lines}
which listed in Table 2.

In Table 1 and Table 2, we renumber these {\sc Lines},
called {\sc Cases}, according to the different parameters $$(d, c, x, y,
\gamma, \delta, k^{(v)}, k^{(r)}, b^{(v)}, b^{(r)})$$ other than the
intersection type, so there are 11 {\sc Cases} in Table 1 and 106
{\sc Cases} in Table 2 needed to analysis.

It is worth mentioning that among the computation of the intersection types, we modified the algorithm by adding the restrictive condition $\frac{\al_i}{gcd(\al, \be)}|d_i'$  of \cite[Prop 4.5(vi)]{BDLPZ}.

\section{The case where $k^{(r)}=10$}

\begin{proposition}
  Let $\S$ be a non-trivial finite linear space with $v$ points
and Fang-Li parameter $k^{(r)}=10$.  Assume that $\S$ has a
subgroup $G$ of automorphism group which is line-transitive and
point-imprimitive. Then $\S$ is the Desarguesian projective plane
$\PG(2,9)$.
\end{proposition}

\begin{proof} We deal with 11 {\sc Cases} in Table 1 case by case.

 {\sc Cases} 1 and 3-5: In each of these {\sc Cases}, $k=10$. Recall that \cite{PZ2007} classified the linear spaces with line size at most 12, so by \cite[Theorem 1.1]{PZ2007}, if {\sc Case} 1 or 4 holds,  then $\S$ is the Desarguesian projective plane $\PG(2,9)$, and
{\sc Cases} 3, 5 can not occur by \cite{GC2001}.

{\sc Case} 2: Assume {\sc Case 2} holds, then  $\Soc(G^C)$ is $\PSL(2,  2^7)$,  or $A_{129}$. It follows that there exists a prime
$p=127\mid |G^C|$, and so $127 \mid|G|$, we also have $p>k=30$, $k^2-k+1>r$,
however $p\nmid v(v-1)=2^3\cdot3^3\cdot 5\cdot 29\cdot 43$, contradicts  Lemma \ref{p-k}. Similarly,
{\sc Case} 8 can not occur, because here $\Soc(G^{\C})$ is $\PSL(2,  2^7)$,  or $A_{129}$, and there exists a prime $p=127$ satisfying the conditions of Lemma \ref{p-k}(ii) but $p\nmid v(v-1)$.

{\sc Cases}  7, 10 and 11:  Since $k^2-k+1>r$, and there exists a
prime $p>k$ which  is listed in Table 1, such that $p\mid |G|$, but
$p\nmid v(v-1)$, a contradiction. In Table 1, the symbol $(T, p)$ or $(B,p)$ denotes the
contradiction by the Top or Bottom group and the prime $p$ such that $p$
divides the order of the top or bottom group but $p$ does not divides
$v(v-1)$.

\medskip
{\sc Cases} 6 and 9:  \textbf{(a)}  For each {\sc Case}, the partition $\C$ is $G$-normal, $K$ is  semiregular.
\medskip

For {\sc Case} 6, suppose first that $G$ is point-quasiprimitive and
almost simple, then $G^{\C} \cong G$ is an almost simple primitive
group of degree 49. As $G$ is point-transitive additionally, we have that $v\mid |G|$. Hence $G\geq A_{49}$. But here $t_{max} =1$, contradicting Lemma \ref{t}. So $\C$ is $G$-normal
relative to $ K$. If $K$ is not semiregular, as $d_1\cdot d_2=14\times 7$ in {\sc Line} 20 and $28\times 7>0$ in {\sc Line} 21, by
\cite [Corollary 4.2]{PraegerTuan} and \cite [Lemma
5.1]{PraegerTuan}, we have
$$1+\frac{r}{k}d_1(d_1-1)\leq |\Fix _{\P}(K_{\al})|\leq d, $$
it follows that $d_1\leq 1$. This is a contradiction because $d_1 = 14$ or $28$ in this {\sc Case}.  Thus $K$ is semiregular, and $K\unlhd G^C$,
 then $K$ is the unique soluble minimal regular normal
subgroup of primitive group $G^C$, so $K=Z_{13}^2$, and $Y\cap
K=K=S, G/Y\leq \Aut(S)=\GL(2, 13)$.

 For {\sc Case} 9, suppose that $G$
is quasiprimitive and almost simple, then $G\cong G^{\C}$ is an
almost simple primitive group of degree 169. Since in {\sc Case} 9,
$t_{max}=2$,   $G$ is at most two transitive. However, checked by GAP(\cite{GAP}),
we know that there is no such group satisfying the conditions. Thus the partition $\C$
 is $G$-normal relative to $K$. Moreover, in {\sc Case} 9, $k=70$ and $x=14$, so we have that
  $$k>2x+\frac{3}{2}+\sqrt{4x-\frac{7}{4}}=\frac{59+\sqrt{217}}{2},$$
 hence by Lemma \ref{Kc}, $K$ is semiregular, and $G$ is affine. Therefore $K=Z_{7}^2,
K=S=Y\cap K=\Soc(G^C), G/Y\leq \Aut(S)=GL(2,7). $

\medskip
\textbf{(b)}  {\sc Case} 9 can not occur.
\medskip

 Suppose
$Y^{\C}=1$, then $Y=Y\cap K=S, Y\leq K$, hence $|\G1|=|G/K|\mid
|G/Y|$. It follows that $d=169\mid (7^2-1)(7^2-7)$, a contradiction.
So, by \cite[Lemma 2.1]{PZ2007}, there exists a normal subgroup $M$
of $G$, such that $S<M\leq Y, M/S\cong M^{\C}$ is a minimal subgroup of $G^{\C}$, and $M=T\times S$.

(1)\, $\Soc(G^{\C})=Z_{13}^2$, then $M^{\C}=Z_{13}^2$. In this case $M$ is regular on
point $\mathcal{P}$, with a unique Sylow 13-subgroup
$T=Z_{13}^2$, and $T$ is regular on $\C$. Since the intersection type is $(1^{42},2^{14})$,  all entries in the
$\C$-intersection type are at most 2, no element of $S$ of order
7 fixes a line, for if $g\in S$, $o(g)=7$ and $g\in G_{\lam}$ for
some $\lam \in \L$, then $g$
 fixes $C\cap \lam $ for any $C \in \C$, particularly, $g$ fixes $C\cap
\lam $ for which $|C\cap \lam|=1$, thus $g$ has fixed points,
contradicting  the fact that $S$ is semiregular.  Similarly there is
no element of order 13 in $T$ fixing a line. Hence $M$ is semirugular
on $\S$. So $v\mid b$. However, in {\sc Case} 9, $\frac{b}{v}=\frac{b^{(r)}}{k^{(v)}}=\frac{12}{7}$, a contradiction.

(2)\, $\Soc(G^{\C})=A_{13}^2$ or $\PSL(3, 3)^2$. Here
$\Soc(G^{\C})\neq A_{169}$ since $t_{max}=2$. In this case
$M^{\C}$ is nonabelian.

(2.1)\,  $\Soc(G^{\C})=A_{13}^2$. Then $\C={\C}_1\times
{\C}_2$, $\C=\{C_{ij}\,|\,1\leq i\leq 13, 1\leq j\leq 13\}$,
$A_{13}^2\leq G^{\C} \leq S_{13}\wr S_2$ in product action.
$\Soc(G^{C})$ acts on $\C$ as follow. Let $(h_1, h_2)\in
A_{13}^2\leq G^{\C}$, and $C_{ij} \in \C$. Then $C_{ij}^{(h_1, h_2)}
=C_{i^{h_1}j^{h_2}}$. Let $P$ be a Sylow 11-subgroup of $G$. Since
$11\nmid b$, $P \leq G_{\lam}$ for some $\lam \in \L$. Moreover,
because $11 \nmid |GL(2, 7)|$, then if $P$ fixes some $C_{ij}$
setwise, it must fix $C_{ij}$ piontwise. Hence $\Fix
_{\mathcal{P}}(P)$ is a union of $\C$-classes. Without loss generality, we assume that
$P=\langle(1,2,3,\cdots, 11)\rangle \times \langle(1,2,3,\cdots,
11)\rangle$, so $\Fix _{\C}(P)=\{C_{ij}|i, j=12\,\, {\mbox{or}}\,\,
13\}$, thus $|\Fix _{\C}(P)|=4$. Furthermore, since the intersection type is
$(1^{42}, 2^{14})$ and $P \leq G_{\lam}$, $P$ fixes setwise the 42
classes $C$ such that $|C\cap \lam |=1$ and the 14 classes $C$
setwise such that $|C\cap \lam |=2$.  Since $42\equiv 9\pmod{11}$ and
$14\equiv 3\pmod{11}$, thus $P$ must fix at least $9+3=12$ classes
setwise, a contradiction.

(2.2)\, $\Soc(G^{\C})=PSL(3, 3)^2$ and $M^{\C}\cong T^{\C}=T$.
For the Schur Multiplier of $\PSL(3, 3)$ is trivial, by Lemma
\ref{lem:KSXY}, either $T$ is transitive on points or the set of
$T$-orbits forms a $G$-normal partition $\C'$ satisfying {\sc Case}
6. In the latter case, $T$ must be abelian, a contradiction.
Thus $T$ is transitive on $\mathcal{P}$, then
 $c\mid |T|$, anther contradiction. Therefore {\sc Case} 9 can not occur. In Table 1 the symbol KSXY refers to the situation that the {\sc Case} is ruled out by Lemma \ref{lem:KSXY}.

\medskip
\textbf{(c)} {\sc Case} 6 can not occur.
\medskip

 Suppose $Y^{\C}=1$, then
$Y=Y\cap K=S, Y\leq K$, hence $|G^{\C}|=|G/K|\mid |G/Y|$ and so
$|\G1|\mid |GL(2, 7)|$. Therefore $d=49\mid (13^2-1)(13^2-13)$, a
contradiction. By Lemma \ref{lem:KSXY}, there exists a normal
subgroup $M$ of $G$, such that $S<M\leq Y, M/S\cong M^{\C}$,
$M=T\times S$, and either $T$ is transitive on $\mathcal{P}$ or $T$
induces a $G$-normal partition of $\mathcal{P}$ with 169 parts of
size 49.  The latter case can not occur because the
corresponding {\sc Case}, i.e., {\sc Case 9}, has been ruled out in (b). Thus $T$ is transitive on $\mathcal{P}$,
then by Lemma \ref{lem:KSXY} $\Soc(G^{\C})=A_{7}^2$ or
$PSL(2, 7)^2$, and $c\mid |T|$. So $\Soc(G^{\C})=A_{7}^2 $. As a result,
$\C={\C}_1\times {\C}_2$, $\C=\{C_{ij}\,|\,1\leq i\leq 7, 1\leq
j\leq 7\}, A_{7}^2\leq \G1 \leq S_{7}\wr S_2$ in product action.
$\Soc(G^{\C})$ acts on $\C$ as follow. Let $(h_1, h_2)\in A_7^2\leq
G^{\C}$, $C_{ij} \in \C$.  Then $C_{ij}^{(h_1, h_2)}
=C_{i^{h_1}j^{h_2}}$. Let $P$ be a Sylow 5-subgroup of $G$. Since
$5\nmid b$, $P \leq G_{\lam}$ for some $\lam \in \L$.  Moreover,
because $5 \nmid |GL(2, 13)|$, then if $P$ fixes some $C_{ij}$
setwise, it must fix $C_{ij}$ piontwise. Hence $\Fix
_{\mathcal{P}}(P)$ is a union of $\C$-classes. We assume that
$P=\langle (12345)\rangle \times \langle(12345)\rangle$, so $\Fix
_{\C}(P)=\{C_{ij}|i, j=6\,\, {\mbox{or}}\,\,7\}$. Thus $|\Fix
_{\C}(P)|=4$. Since the intersection type is
$(1^{14}, 2^{2}, 3^{14})$ in {\sc Line} 20 and $(1^{28}, 2^{7}, 4^{7})$ in {\sc Line} 21,
and $P \leq G_{\lam}$, similar to the proof of {\sc Case} 9, we can
rule out these two possibilities.

This completes the proof of Proposition 5.1.
\end{proof} $\hfill\square$

\section{The case where $k^{(r)}=9$}
\begin{proposition}
  There is no non-trivial line-transitive point-imprimitive  finite linear space $\S$ with the Fang-Li parameter $k^{(r)}$ is $9$.
\end{proposition}
\begin{proof} Suppose for the contrary that there are line-transitive point-imprimitive symmetric designs with $k^{(r)}$ is $9$. Then we need to deal with
 106 {\sc Cases} listed in Table 2.  Here the unstated notations in the last collum of Table 2 please refer to Table 1.
We  prove Proposition 6.1 in 7 steps.

\medskip
\textbf{Step 1.} Rule out 66 {\sc Cases}.
\medskip

There are 6 {\sc Cases} satisfing $k\leq12$ and ruled out by
\cite[Theorem 1.1]{PZ2007},  54 {\sc Cases} ruled out by Lemma
\ref{p-k}. For each of these 54 {\sc Cases}, we have $k^2-k+1>r$ and there exists a prime
$p>k$ such that
$p\mid |G|$, but $p\nmid v(v-1)$ which contracts  Lemma \ref{p-k}. The prime $p$ will be listed in the last column of Table 2.

For {\sc Cases} 86, 87, 89, 92, 104 and  105, we have $G^{\C}\geq A_d$,
but $t_{max}=1$ or 2, a contradiction. The symbol $(T,t_{max})$
denotes this contradiction.

\medskip
\textbf{Step 2.} {\sc Cases} 5, 38, 40, 59, 73, 91, 96 and 106 can
not occur.
\medskip

In {\sc Cases} 40, 91 and 96, we have $\de=2$. By Lemma \ref{1231},
we know that $G^{\C}$ is 2-transitive. In {\sc Cases} 5, 38, 59, 73
and 106, $\de =1$, then by Lemma \ref{123}, $G^{\C}$ is
2-transitive. Therefore, by the classification theorem of
2-transitive permutation groups (\cite[Theorem 5.3]{P1989}), in each
{\sc Case} we have that $\G1\geq A_d$. But $t_{max}=2$,
contradicting Lemma \ref{t}. Therefore these {\sc Cases} are ruled
out. The symbol $(\de, t_{max})$ in Table 2 refers to this
situation.

\medskip
\textbf{Step 3.} {\sc Cases} 10, 14, 17, 23, 29, 33, 34, 44, 46, 49,
50, 63 and  64 can not occur.
\medskip

In {\sc Cases} 10, 14 and 17,  we have $\ga=1$, by Lemma \ref{123},
$G^C$ is 2-transitive.  In {\sc Case} 46, $\ga=2$, and $d$ is not a
prime power, by Lemma \ref{1231}, $G^C$ is 2-transitive. In {\sc
Case} 23, $\ga=1$, by Lemma \ref{123}, $G^C$ is also 2-transitive.
So by \cite[Theorem 5.3]{P1989}, for every {\sc Case} the bottom
group $G^C\geq A_c$.

For  {\sc Cases} 10, 14 and 19, we have $k^2-k+1>r$, and a prime $p>k$ which is listed in Table 2,  with $p\mid |G^C|$, but $p\nmid v(v-1)$, contradicting Lemma \ref{p-k}.

For any other {\sc Case}, there is a prime $p$ which is listed in Table 2, such that $p$ divides
  $|G|$,  but does not divides $b$.  Thus, in each {\sc Case},  a Sylow $p$-subgroup $P$ of $G$ will
fix some line $\lam \in \L$. Let $F = \Fix_{\P}(P)$. For the line
size $k$, let $k_1$ to be the integer such that $0 \leq k_1 < k$ and
$k\equiv k_1\pmod p$. In each {\sc Case} $k_1>2$ and $(v-k)\nmid p$
guarantee that there is at least one point which is not on $\lam$ fixed by
$P$.
 Then by Lemma \ref{ngh}, $F$ induces an linear space $\S|_F=(F, \L|_F)$. For this subdesign,
 the number of the points $v_0=|F|\leq k+r-p-1$, the line size $k_0=|\lam \cap F|\geq k_1$, since $v_0-1\geq k_0(k_0-1)\geq k_1(k_1-1)$,
 then $$k_1(k_1-1)\leq k+r-p-1.$$
 But this inequality does not hold in every {\sc Case}, the desired contradiction.

\medskip
\textbf{Step 4.} {\sc Cases}  22 and 30 can not occur.
\medskip

{\sc Case} 22:  Since $G^C\geq A_{1370}$ or $\PSL(2, 37^2)$. Let
$p=1367$, $P\in Syl_{p}(G)$. If $G^C\geq A_{1370}$, then as above
paragraph, we can get the desired contradiction. So $G^C \geq
\PSL(2, 37^2)$. Since $t_{max}=1$, we have $G^{\C}\ngeq A_{17798}$.
For $d=17798=22\times 809$, let $p=809$, then $G$ has an element $g$
of order 809. Suppose that $g$ has at least one fixed point in $\C$,
then it has at least 809 fixed points and $q$ cycles of length $p$
where $q\leq21$. By \cite[Theorm 13.10]{Hw}, $|\Fix_{\C}(g)|\leq
4q-4$, a contradiction. So $g$ has no fixed point on $\C$. Thus, by
\cite[Therom 1.1, Table 3]{mlj}, there is no such primitive group, a
contradiction.

{\sc Case} 30:  Since $G^C\geq A_{1070}$ or $\PSL(2, 1069)$. Let
$p=1063$, $P\in Syl_p(G)$. If $G^C\geq A_{1070}$, similar to Step 3, we can get the desired contradiction. Thus $G^C \geq
\PSL(2, 1069)$. Because the prime $p=1069\nmid d$, there is a part
$C\in\C$, such that $P\leq G_C$, so $P^C<G^C\leq \PGL(2, 1069)$.
Since $\PGL(2, 1069)$ is sharply 3-transitive, the nontrivial
subgroup $P^C$ can fix at most 2 points of $C$. But according to
Lemma \ref{p3}, $|\Fix_C(P)|\geq 3$, a contradiction.

\medskip
\textbf{Step 5.} {\sc Cases} 1, 2,  4, 6, 7, 11, 27, 32, 47 and  65
can not occur.
\medskip

For each {\sc Case}, we have that
$k>2x+\frac{3}{2}+\sqrt{4x-\frac{7}{4}}$. If the partition is $G$-normal, by Lemma \ref{Kc}, then $c$ is
a prime power, a contradiction. So $G^{\C}\cong G$ is quasiprimitive
and almost simple.

For {\sc Cases} 1, 2,  4, 7, 11, 32, 47 and 65, since $\lcm(b, v)\mid
|G|$, then $G\geq A_d$, but in each {\sc Case},  $t_{max}=1$ or 2,
contradicting   Lemma \ref{t}. For {\sc Case} 6, we have $\de=1$, by
Lemma \ref{123}, $G^{\C}$ is 2-transitive. Then by \cite{P1989}, we
can find that $G^{\C}\geq A_{8090}$ or $\PSL(2, 8089)$. Furthermore,
$b\mid |G|$, so $G^{\C}\geq A_{8090}$, as a result, the prime
$8087\mid|G|$, this contradicts  Lemma \ref{p-k}. For{\sc Case}
27, since $t_{max}=2$, then $\Soc(G^{\C})=\PSL(5, 5)$, it follows
that $G_C=G^{\C}_C=5^4.\GL(4, 5)$. However, from
$\Soc({G^C})=\PSL(4, 3), \PSU(4, 2)$ or $A_{40}$, we have
$3^3\mid|G_C|$, this is  a contradiction.

\medskip
\textbf{Step 6.} {\sc Cases}  35, 51, 56, 69 and 97 can not occur.
\medskip

(1)\, The partition of each {\sc Case} is $G$-normal. Otherwise,
$G^{\C}\cong G$ is an almost simple primitive group of degree $d$,
then $[v, b]\mid |G^{\C}|$, there is no such group in {\sc Case
35}, 56 or 97. Now $G\geq A_{103}$ in {\sc Case} 51, but here
$t_{max}$ is 2, a contradiction. So the partition is $G$-normal
relative to $K$.

(2)\, Since $x<8$ and $\C$ is $G$-normal in {\sc Case} 51, it is
ruled out by \cite[Theorem 1.6]{PraegerTuan}.
 For {\sc Case} 69, we have that
$k>2x+\frac{3}{2}+\sqrt{4x-\frac{7}{4}}$.
 By Lemma \ref{Kc},  $K=Z_{103}$ is
 semiregular, and $Y\cap K=K=S, G/Y\leq \Aut(S)=Z_{102}$. Since $52\mid|G|$, but $52\nmid 102$, thus $52\mid|Y|$, hence $Y^{\C}\neq 1$.
Therefore, by Lemma \ref{lem:KSXY},  there exists a subgroup $M$ of
$G$, such that $S<M\leq Y, M/S\cong M^{\C}$, and $M=T\times S$.  By
\cite[Table B.4]{DM1996}, $ \Soc(G^{\C})=\PSL(3, 3).2$ or $
A_{52}$. Also we have  that $(c, d)=1$ and by \cite[Theorem
5.1.4]{KL1990}, the Schur multiplier of $\PSL(3, 3)$, $A_{52}$ has
no section isomorphic to $S$, so either $T$ is transitive on
$\mathcal{P}$ or $T$ induces a new partition of $\mathcal{P}$ with
103 parts of size 52. The latter case can not occur because there is
no corresponding {\sc Case} with $(d,c)=(52,103)$ in Table 2. Thus
$T$ is transitive on $\mathcal{P}$, and $v \mid |T|$, a
contradiction.  {\sc Cases} 35 and 56 can be ruled out by the
similar method.

For {\sc Case} 97, if $K$ is not semiegular, by  \cite [Corollary
4.2]{PraegerTuan} and \cite [Lemma 5.1]{PraegerTuan}, we have that
$$1+\frac{r}{k}d_1(d_1-1)\leq |\Fix_{\P}(K_{\al})|\leq d, $$
a contradiction.Thus  $K$ is semiregular. Since $ K\unlhd G^C$,
then $K$ is the unique soluble minimal regular normal subgroup of
primitive group $G^C$, hence $K=Z_{257}$, and $Y\cap K=K=S, G/Y\leq
\Aut(S)=Z_{256}$. Since $13\mid|G^{\C}|, 13\nmid |G/Y|$, so
$Y^{\C}\neq1$. By Lemma \ref{lem:KSXY},  there exists a normal
subgroup $M$ of $G$, such that $S<M\leq Y, M/S\cong M^{\C}$,
$M=T\times S,$ by \cite[Tble B.4]{DM1996}, $\Soc(G^{\C})=A_{65}$, $\PSL(2, 2^6)$, $\PSL(2, 5^2)$,
$\PSU(3, 2^2)$ or $Sz(2^3)$, the condition of Lemma \ref{lem:KSXY} holds by \cite[Theorem 5.1.4]{KL1990}, therefore either $T$ is transitive on $\mathcal{P}$ or $T$
induces a $G$-normal partition of $\mathcal{P}$ with 257 parts of
size 40. The latter case can not occur because there is no
corresponding {\sc Case}. Thus $T$ is transitive on $\mathcal{P}$,
and $v\mid |T|$,
which is impossible.

\medskip
\textbf{Step 7.} {\sc Cases}  85 and 101 can not occur.
\medskip

{\sc Case} 101: Since $\de=1$, then $G^{\C}$ is 2-transitive, hence
$\Soc(G^{\C})=\PSL(4, 3)$ or $A_{40}$. If $G\cong G^{\C}$, then
$71||G^{\C}|$, a contradiction. So the partition is $G$-normal. Suppose
$K$ is semiregular, then $c$ must be a prime power, this
contradiction leads to that $K$ is not semiregular. By \cite[Prop.
 4.3]{PraegerTuan},
$$d_1\leq\frac{1}{2}+\sqrt{2k-\frac{7}{4}}.$$
So in this {\sc Case} the intersection type must be $(2^{10},
3^{10}, 4^{10})$, i.e., $d_1=0$. Now by \cite[Prop. 5.2(c)]{PraegerTuan}, {\sc
Case} 101 is ruled out.

{\sc Case} 85:  It can be ruled out by the similar method of {\sc
Case} 101.

This complete the proof of Proposition 6.1. $\hfill\square$

\end{proof}

\begin{corollary}
  Let $\S$ be a non-trivial finite linear space with the Fang-Li parameter $k^{(r)}$ is $9$.  If  $G^{\C}\leq Aut({\cal S})$
is line-transitive, then $G$ is  point-primitive.
\end{corollary}

\medskip

{\bf Proof of Theorem 1.1}\,\, This follows from Propositions 5.1
and 6.1. $\hfill\square$

\newpage
\section{Appendix}

\begin{table*}[!h]
\begin{center}

Table 1: Potential parameter sequences for $k^{(r)}=10$

\bigskip
\begin{footnotesize}

\begin{tabular}{ccccccclcl}
\hline {\sc Line}&{\sc Case} & $d\cdot c$ & $(x,y)$ & $(\gamma,\delta)$ &
$k^{(v)}\cdot k^{(r)}$ & $b^{(v)}\cdot b^{(r)}$ & int type &
$t_{\max}$ &Reference\cr \hline
1&1& $7\cdot13$   & $(6,3)$ & $(12,6)$ & $1\cdot10$ & $91\cdot1$   & $(1,2^3,3)$   &1&$\PG(2,9)$, \cite{PZ2007}\\

2& & $7\cdot13$   & $(6,3)$ & $(12,6)$ & $1\cdot10$ & $91\cdot1$   & $(1^4,3^2)$   &1&\\

3&& $7\cdot13$   & $(6,3)$ & $(12,6)$ & $1\cdot10$ & $91\cdot1$   & $(1^6,4)$   &7&\\

4&2& $9\cdot129$ & $(48,3)$ & $(32,2)$ & $3\cdot10$ & $387\cdot4$ & $(3^6,6^2)$ &2&(B, 127)\\
5&& $9\cdot129$ & $(48,3)$ & $(32,2)$ & $3\cdot10$ & $387\cdot4$ & $(2^9, 3, 9)$ &2&\\
6&& $9\cdot129$ & $(48,3)$ & $(32,2)$ & $3\cdot10$ & $387\cdot4$ & $(1^3,4^3,5^3)$ &2&\\
7&& $9\cdot129$ & $(48,3)$ & $(32,2)$ & $3\cdot10$ & $387\cdot4$ & $(1^3, 2^3, 5^3, 6)$ &2&\\
\hline
\end{tabular}\end{footnotesize}
\end{center}
\end{table*}

\newpage

 \begin{center}
   Table 2:  Potential parameter sequences for $k^{(r)}=9$\medskip
\begin{tiny}
\begin{supertabular}{ccccccclcl}

\hline \tablefirsthead{ \hline {\sc Line}& {\sc Case}& $d\cdot c$ & $(x,y)$ &
$(\gamma,\delta)$ & $k^{(v)}\cdot k^{(r)}$ & $b^{(v)}\cdot b^{(r)}$
& int type & $t_{\max}$&Reference\\\hline}

\tablehead{ \hline {\sc Line}& {\sc Case}& $d\cdot c$ & $(x,y)$ &
$(\gamma,\delta)$ & $k^{(v)}\cdot k^{(r)}$ & $b^{(v)}\cdot b^{(r)}$
& int type & $t_{\max}$&Reference\\\hline}

\tabletail{%
\hline
\multicolumn{10}{r}{\small\sl continued on next page}\\
\hline} \tablelasttail{\hline} {\sc Line}& {\sc Case}& $d\cdot c$ & $(x,y)$ &
$(\gamma,\delta)$ & $k^{(v)}\cdot k^{(r)}$ & $b^{(v)}\cdot b^{(r)}$
& int type & $t_{\max}$&Reference\\\hline

1&1 &  $152\cdot  152$  &   $(1, 1)$    &    $(1, 1)$    &  $2\cdot  9$   &   $11552\cdot  151$    &   $({1}^{16}, {2})$     &2 & Step 5\\
2&  2     &     $314\cdot  314$     &     $(2, 2)$     &     $(1, 1)$     &     $4\cdot  9$     &     $24649\cdot  313$     &     $({1}^{32}, {2}^{2})$  &2 & Step 5 \\
3    & 3    &     $800\cdot  800$     &     $(5, 5)$     &     $(1, 1)$     &     $10\cdot  9$     &     $64000\cdot  799$     &     $({1}^{80}, {2}^{5})$ &2&(T,797)\\
4& &     $800\cdot  800$     &     $(5, 5)$     &     $(1, 1)$     &     $10\cdot  9$     &     $64000\cdot  799$     &     $({1}^{83}, {2}^{2}, {3})$ &2 &\\
 5&    4    &     $1610\cdot  1610$     &     $(10, 10)$     &     $(1, 1)$     &     $20\cdot  9$     &     $129605\cdot  1609$     &     $({1}^{160}, {2}^{10})$     &2 &  Step 5 \\
6&         &     $1610\cdot  1610$     &     $(10, 10)$     &     $(1, 1)$     &     $20\cdot  9$     &     $129605\cdot  1609$     &     $({1}^{166}, {2}^{4}, {3}^{2})$     &2& \\
7&5   &     $4040\cdot  4040$     &     $(25, 25)$     &     $(1, 1)$     &     $50\cdot  9$     &     $326432\cdot  4039$     &     $({1}^{400}, {2}^{25})$     &2   &($\delta$, $t_{max})$\\
\end{supertabular}
\end{tiny}
 \end{center}

\end{document}